\documentclass{icmart}

\usepackage{epsfig}

\contact[mbel@impa.br]{IMPA\\ Estrada Dona Castorina, 110 \\ 22460-320 Rio de Janeiro, Brazil}

\newtheorem{theorem}{Theorem}[section]
\newtheorem{corollary}[theorem]{Corollary}
\newtheorem{lemma}[theorem]{Lemma}
\newtheorem{proposition}[theorem]{Proposition}
\newtheorem{conjecture}[theorem]{Conjecture}

\theoremstyle{definition}

\newcommand{\Vol}{\operatorname{Vol}}
\newcommand{\dist}{\mathrm{dist}}

\def\Hy{\mathbf{H}^n}
\def\Sph{\mathbf{S}^n}
\def\cM{\mathcal{M}}
\def\cO{\mathcal{O}}
\newcommand\Pa{\mathrm P}
\def\cE{\mathcal{E}}

\newcommand{\ZZ}{\mathbb Z}
\newcommand{\QQ}{\mathbb Q}
\newcommand{\RR}{\mathbb R}

\newcommand{\Adels}{\mathbb A}

\newcommand{\disc}{\mathcal{D}}
\newcommand{\Vf}{V_{\mathrm{f}}}

\DeclareMathOperator{\Spin}{Spin}
\DeclareMathOperator{\SO}{SO}
\DeclareMathOperator{\PO}{PO}

\DeclareMathOperator{\PSL}{PSL}

\newcommand{\GG}{\mathbf G}

\newcommand\An{\mathrm{A}}
\newcommand\Bn{\mathrm{B}}

\newcommand\Dn{\mathrm{D}}

\title[Hyperbolic orbifolds of small volume]{Hyperbolic orbifolds of small volume}

\author[Mikhail Belolipetsky]
{Mikhail Belolipetsky \thanks{The author is supported by a CNPq research grant.}}

\begin{document}

\begin{abstract}
Volume is a natural measure of complexity of a Riemannian manifold. In this survey, we discuss the results and conjectures concerning $n$-dimensional hyperbolic manifolds and orbifolds of small volume.
\end{abstract}



\maketitle

\section{Volume in hyperbolic geometry}\label{mbel:sec1}

A hyperbolic manifold is an $n$-dimensional manifold equipped with a complete Riemannian metric of constant sectional curvature $-1$. Any such manifold $\cM$ can be obtained as the quotient of the hyperbolic $n$-space $\Hy$ by a torsion-free discrete group $\Gamma$ of isometries of $\Hy$:
$$\cM = \Hy/\Gamma.$$
If we allow more generally the discrete group to have elements of finite order, then the resulting quotient space $\cO = \Hy/\Gamma$ is called a \emph{hyperbolic $n$-orbifold}.

We can descend the volume form from $\Hy$ to $\cO$ and integrate it over the quotient space. This defines the hyperbolic volume of $\cO$.
The generalization of the Gauss-Bonnet theorem says that in even dimensions the volume is proportional to the Euler characteristic. More precisely, we have for $n$ even:
\begin{equation}
 \Vol(\cM) = \frac{\Vol(\Sph)}{2} \cdot (-1)^{n/2} \chi(\cM),
\end{equation}
where $\Vol(\Sph)$ is the Euclidean volume of the $n$-dimensional unit sphere and $\chi(\cM)$ denotes the Euler characteristic. This formula generalizes to hyperbolic $n$-orbifolds with the orbifold Euler characteristic in place of $\chi$. Conceptually it says that the hyperbolic volume is a topological invariant and, like for the Euler characteristic, its value is a measure of complexity of the space. In odd dimensions the Euler characteristic vanishes but the volume is still a non-trivial topological invariant that measures the complexity of $\cM$. Indeed, the Mostow--Prasad rigidity theorem implies that every geometric invariant of a finite volume hyperbolic $n$-manifold (or orbifold) of dimension $n \ge 3$ is a topological invariant. One particular example of an application of the volume of the hyperbolic $3$-manifolds as a measure of complexity appears in knot theory --- see \cite{CDW99} and \cite{CKP04} where all knots up to a certain complexity are enumerated. Note that although many knots in the tables of Callahan--Dean--Weeks and Champanerkar--Kofman--Patterson have the same number of simplexes in the minimal triangulations, only very few of them share the same volume. For large volume the picture is different (see \cite{Millichap2013} for the recent results in dimension $3$ and \cite{Eme12} for higher dimensions and other symmetric spaces), but nevertheless, in practice, hyperbolic volume has proven to be very effective in distinguishing manifolds.

The main purpose of this report is to discuss what is currently known about the simplest (i.e. \emph{minimal volume}) hyperbolic $n$-manifolds and orbifolds. More information about this topic with a particular emphasis on a connection with hyperbolic reflection groups can be found in a recent survey paper by Kellerhals \cite{Kel14}.

The minimal volume problem for hyperbolic $n$-orbifolds goes back to the paper of Siegel \cite{Sieg45} where the general setup is described and the solution to the problem for $n = 2$ is given. In fact, the solution of the $2$-dimensional problem can be traced back to the earlier work of Hurwitz \cite{Hurw1893}, which is briefly mentioned in Siegel's paper. The qualitative solution to Siegel's problem in general was obtained by Kazhdan and Margulis in \cite{KM68} (the title of \cite{KM68} refers to a conjecture of Selberg about the existence of unipotent elements in non-uniform lattices which was also resolved in the same $5$-page paper). We are going to come back to the discussion of the Kazhdan--Margulis theorem in Section~\ref{mbel:sec5}, but before that we shall consider the sharp lower bounds for the volumes of \emph{arithmetic orbifolds}.

\medskip

\noindent{\bf Acknowledgements.}
I would like to thank Vincent Emery for the fruitful collaboration throughout the years which gave rise to an essential part of the results that are considered in this report. I thank Misha Kapovich for allowing me to include Proposition~\ref{prop_mu_n}. I also thank Matthew Stover for the comments on a preliminary version of the paper.

\section{Arithmeticity and volume}\label{mbel:sec2}

The group of isometries of the hyperbolic $n$-space is isomorphic to the real Lie group $\PO(n,1)$. Its subgroup of orientation preserving isometries corresponds to the identity component $H = \PO(n,1)^\circ$, which can be further identified with the matrix group $\SO_0(n,1)$~--- the subgroup of $\SO(n,1)$ that preserves the upper half space. We shall mainly consider \emph{orientable finite volume hyperbolic $n$-orbifolds}
$$ \cO = \Hy/\Gamma, \quad \Gamma \text{ is a lattice in } H.$$

Let $\GG$ be an algebraic group defined over a number field $k$ which admits an epimorphism $\phi: \GG(k\otimes_\QQ \RR)^\circ \to H$ whose kernel is compact. Then, by the Borel--Harish-Chandra theorem \cite{BorHC62}, $\phi(\GG(\cO_k))$ is a finite covolume discrete subgroup of $H$ (here and further on $\cO_k$ denotes the ring of integers of $k$). Such subgroups and all the subgroups of $H$ which are commensurable with them are called \emph{arithmetic lattices} (or \emph{arithmetic subgroups}), and the field $k$ is called their \emph{field of definition}.

It can be shown that to define all arithmetic subgroups of $H$ it is sufficient to consider only simply connected, absolutely simple $k$-groups $\GG$ of absolute type $\Bn_{n/2}$, if $n$ is even, or $\Dn_{(n+1)/2}$, if $n$ is odd. In this case $\GG(k\otimes_\QQ \RR) \cong \widetilde{H}\times K$, where $\widetilde{H} = \Spin(n,1)$ is the simply connected covering of $H$ and $K$ is a compact Lie group. We shall call such groups $\GG$ and corresponding fields $k$ \emph{admissible}. The Godement compactness criterion implies that for $n \ge 4$ the quotient $\Hy/\Gamma$ is noncompact if and only if it is defined over $k = \QQ$.

From the classification of semisimple algebraic groups \cite{Tits66} it follows that if $n$ is even then $\GG$ has to be the spinor group of a quadratic form of signature $(n,1)$ defined over a totally real field $k$, i.e. in even dimensions the arithmetic subgroups are commensurable with the groups of units of the quadratic forms. For odd $n$ there is another family of arithmetic subgroups corresponding to the groups of units of appropriate Hermitian forms over quaternion algebras. Moreover, if $n = 7$ there is also the third type of arithmetic subgroups of $H$ which are associated to the Cayley algebra.

The number theoretic local-global principle gives us a way to construct arithmetic lattices that is particularly suitable for the volume computations. Let $\Pa = (\Pa_v)_{v\in\Vf}$ be a collection of parahoric subgroups $\Pa_v \subset \GG(k_v)$, where $v$ runs through all finite places of $k$ and $k_v$ denotes the non-archimedean completion of the field (see e.g. \cite[Sec.~0.5]{Pra89} for the definition of parahoric subgroups).
The family $\Pa$ is called \emph{coherent} if $\prod_{v\in\Vf} \Pa_v$ is an open subgroup of the finite ad\`ele group $\GG(\Adels_f(k))$. Following \cite{Pra89}, the group
$$\Lambda = \GG(k) \cap \prod_{v\in\Vf} \Pa_v$$
is called the \emph{principal arithmetic subgroup} of $\GG(k)$ associated to $\Pa$.  We shall also call $\Lambda' = \phi(\Lambda)$ a principal arithmetic subgroup of $H$. This construction is motivated by a simple observation that the integers $\ZZ = \QQ \cap \prod_{p\text{ prime}}\ZZ_p$, where $\QQ$ is embedded diagonally into the product of $p$-adic fields $\prod_{p\text{ prime}}\QQ_p$, and can be understood as its generalization to the algebraic groups defined over number fields. We refer to the books by Platonov--Rapinchuk \cite{PlaRap94} and Witte Morris \cite{WittMorr08} for more material about arithmetic subgroups and their properties.

The Lie group $H$ carries a Haar measure $\mu$ that is defined uniquely up to a scalar factor. We can normalize $\mu$ so that the hyperbolic volume satisfies
$$\Vol(\Hy/\Gamma) = \mu(H/\Gamma).$$
The details of this normalization procedure are explained, for instance, in Section~2.1 of \cite{BelEme}. If $\Gamma$ is a principal arithmetic subgroup, its covolume can be effectively computed. The first computations of this kind can be traced back to the work of Smith, Minkowski and Siegel on masses of lattices in quadratic spaces. After the work of Kneser, Tamagawa and Weil these computations were brought into the framework of algebraic groups and number theory. More precisely, if $\Gamma$ is an arithmetic subgroup of $\GG$ defined over $k$, then its covolume can be expressed through the volume of $\GG(\Adels_k)/\GG(k)$ with respect to a volume form $\omega$ associated naturally to $\Gamma$, and one can relate $\omega$ to the Tamagawa measure of $\GG(\Adels_k)$ by virtue of certain local densities. Assuming that the Tamagawa number of $\GG$ is known, the computation of the covolume of $\Gamma$ is thus reduced to the computation of these local densities. The precise expressions of this form are known as the \emph{volume formulas}, among which we would like to mention the Gauss--Bonnet formula of Harder \cite{Harder71}, Borel's volume formula \cite{Bor81}, Prasad's formula \cite{Pra89}, and its motivic extension by Gross \cite{Gross97}. In his paper, Harder worked out an explicit formula for the split groups $\GG$ but in our case, if $n>3$, the corresponding algebraic groups are never split. In Borel's influential paper the case of the semisimple groups of type $\An_1$ is covered in full generality. This corresponds to the hyperbolic spaces of dimensions $2$ and $3$ and their products. Our primary interest lies in higher dimensions, where the computations can be carried out via Prasad's volume formula.

\medskip

Let us recall \emph{Prasad's formula} adapted to our setup. Let $\Lambda$ be a principal arithmetic subgroup of an admissible group $\GG/k$ associated to a coherent collection of parahoric subgroups~$\Pa$. Following \cite[Section 2.1]{BelEme}, assuming $\Lambda$ does not contain the center of $\GG$, we have
\begin{equation}
\mu(H/\Lambda')  = \Vol(\Sph) \cdot \disc_k^{\frac12\mathrm{dim}(\GG)}\left(\frac{\disc_\ell}{\disc_k^{[\ell:k]}}\right)^{\frac12s} \left(\prod_{i=1}^{r}\frac{m_i!}{(2\pi)^{m_i+1}}\right)^{[k:\QQ]} \tau_k(\GG)\:\cE(\Pa),
\end{equation}
where
\begin{itemize}
\smallskip\item[(i)] $\displaystyle\Vol(\Sph) = \frac{2\pi^{\frac{n+1}{2}}}{\Gamma(\frac{n+1}{2})}$ is the volume of the unit sphere in $\RR^{n+1}$;
\smallskip\item[(ii)] $\disc_K$ denotes the absolute value of the discriminant of the number field $K$;
\smallskip\item[(iii)] $\ell$ is a Galois extension of $k$ defined as in \cite[0.2]{Pra89}
(if $\GG$ is not a $k$-form of type $^6\Dn_4$, then $\ell$ is the splitting field of
the quasi-split inner $k$-form of $\GG$, and if $\GG$ is of type $^6\Dn_4$, then
$\ell$ is a fixed cubic extension of $k$ contained in the corresponding splitting
field; in all cases $[\ell:k]\le 3$);
\smallskip\item[(iv)] $\mathrm{dim}(\GG)$, $r$ and $m_i$ denote the dimension, rank and Lie exponents of $\GG$:
\begin{itemize}
 \item[-] if $n$ is even, then $r = \frac{n}2$, $\mathrm{dim}(\GG) = 2r^2+r$, and $m_i = 2i-1$ ($i = 1, \ldots, r$);
 \item[-] if $n$ is odd, then $r = \frac12(n+1)$, $\mathrm{dim}(\GG) = 2r^2-r$, and $m_i = 2i-1$ ($i = 1, \ldots, r-1$), $m_r = r-1$;
\end{itemize}
\smallskip\item[(v)] $s = 0$ if $n$ is even and $s = 2r-1$ for odd dimensions (cf. \cite[0.4]{Pra89});
\smallskip\item[(vi)] $\tau_k(\GG)$ is the Tamagawa number of $\GG$ over $k$ (since $\GG$ is simply connected and $k$ is a number field, $\tau_k(\GG) = 1$); and
\smallskip\item[(vii)] $\cE(\Pa) = \prod_{v\in V_f} e_v$ is an Euler product of the local densities $e_v = e(\Pa_v)$ which can be explicitly computed using Bruhat--Tits theory.
\end{itemize}
When $\Lambda$ contains the center of $\GG$ its covolume is twice the above value.

In even dimensions the right-hand side of the volume formula is related to the generalized Euler characteristic of the quotient (cf. \cite[Section~4.2]{BorPra89}) and we obtain a variant of the classical Gauss--Bonnet theorem.

\medskip

If $\cO = \Hy/\Gamma$ is a minimal volume hyperbolic orbifold then $\Gamma$ is a \emph{maximal lattice} in $H$. It is known that any maximal arithmetic subgroup $\Gamma$ can be obtained as the normalizer in $H$ of some principal arithmetic subgroup $\Lambda$, and that the index $[\Gamma:\Lambda]$ can be evaluated or estimated using Galois cohomology. We refer for more details and some related computations to the corresponding sections of \cite{Bel04}, \cite{Bel07} and \cite{BelEme}. The upshot is that this technique allows us to study the minimal volume arithmetic hyperbolic $n$-orbifolds using volume formulas.

\section{Minimal volume arithmetic hyperbolic orbifolds}\label{mbel:sec3}

The minimal volume $2$-orbifold corresponds to the Hurwitz triangle group $\Delta(2,3,7)$ (cf. \cite{Sieg45}). This group is arithmetic and defined over the cubic field $k = \QQ[\cos(\frac{\pi}{7})]$, which follows from Takeuchi's classification of arithmetic triangle groups \cite{Takeuchi77}. The smallest noncompact $2$-orbifold corresponds to the modular group $\PSL(2, \ZZ)$. In dimension $3$ the minimal covolume arithmetic subgroup was found by Chinburg and Friedman \cite{ChinbFried86}, who used Borel's volume formula \cite{Bor81}. Much later Gehring, Martin and Marshall showed that this group solves Siegel's minimal covolume problem in dimension $3$ \cite{GehrMart09, MarshMart12}. The noncompact hyperbolic $3$-orbifold of minimal volume was found by Meyerhoff \cite{Meyerh85}, it corresponds to the arithmetic Bianchi group $\PSL(2, \cO_3)$, with $\cO_3$ the ring of integers in $\QQ[\sqrt{-3}]$. For even dimensions $n \ge 4$ the minimal volume problem for arithmetic hyperbolic $n$-orbifolds was solved in my paper \cite{Bel04} (with addendum \cite{Bel07}). The odd dimensional case of this problem for $n \ge 5$ was studied by Emery in his thesis \cite{EmePhD} and appeared in our joint paper \cite{BelEme}. These results complete the solution of Siegel's problem for arithmetic hyperbolic $n$-orbifolds. We shall now review our work and discuss some corollaries.

The main results of \cite{Bel04}, \cite{Bel07}, \cite{BelEme} and \cite{EmePhD} can be summarized in two theorems:

\begin{theorem}
For every dimension $n \ge 4$, there exists a unique orientable compact arithmetic hyperbolic $n$-orbifold $\cO^n_0$ of the smallest volume. It is defined over the field $k_0 = \QQ[\sqrt{5}]$ and has $\Vol(\cO^n_0) = \omega_{c}(n)$.
\end{theorem}

\begin{theorem}
For every dimension $n \ge 4$, there exists a unique orientable noncompact arithmetic hyperbolic $n$-orbifold $\cO^n_1$ of the smallest volume. It is defined over the field $k_1 = \QQ$ and has $\Vol(\cO^n_1) = \omega_{nc}(n)$.
\end{theorem}

The values of the minimal volume are as follows:
\begin{equation}\label{mbel:sec3:f1}
\omega_{c}(n) = \left\{
\begin{array}{ll}
\frac{4 \cdot 5^{r^2+r/2} \cdot (2\pi)^r}{(2r-1)!!} \prod_{i=1}^{r}\frac{(2i-1)!^2}{(2\pi)^{4i}}\zeta_{k_0}(2i),
&\text{if } n= 2r,\ r \text{ even};\\
\frac{2 \cdot 5^{r^2+r/2} \cdot (2\pi)^r \cdot (4r-1)}{(2r-1)!!} \prod_{i=1}^{r}\frac{(2i-1)!^2}{(2\pi)^{4i}}\zeta_{k_0}(2i),
&\text{if } n= 2r,\ r \text{ odd};\\
\frac{5^{r^2-r/2} \cdot 11^{r-1/2} \cdot (r-1)!}{2^{2r-1} \pi^r} \; L_{\ell_0|k_0}\!(r) \; \prod_{i=1}^{r-1} \frac{(2i -1)!^2}{(2\pi)^{4i}} \zeta_{k_0}(2i),
&\text{if } n= 2r - 1;\\
\end{array}
\right.
\end{equation}

\begin{equation}\label{mbel:sec3:f2}
\omega_{nc}(n) = \left\{
\begin{array}{ll}
\frac{4 \cdot (2\pi)^r}{(2r-1)!!} \prod_{i=1}^{r}\frac{(2i-1)!}{(2\pi)^{2i}}\zeta(2i),
& \text{if } n= 2r,\ r\equiv 0,\;1 \text{(mod\ 4)};\\
\frac{2 \cdot (2^r-1) \cdot (2\pi)^r}{(2r-1)!!} \prod_{i=1}^{r}\frac{(2i-1)!}{(2\pi)^{2i}}\zeta(2i),
& \text{if } n= 2r,\ r\equiv 2,\;3 \text{(mod\ 4)};\\
\frac{3^{r-1/2}}{2^{r-1}} \; L_{\ell_1|\QQ}\!(r) \; \prod_{i=1}^{r-1} \frac{(2i -1)!}{(2 \pi)^{2i}} \zeta(2i),
& \text{if } n= 2r-1,\ r \text{ even};\\
\frac{1}{2^{r-2}} \; \zeta(r) \; \prod_{i=1}^{r-1} \frac{(2i -1)!}{(2\pi)^{2i}} \zeta(2i),
& \text{if } n= 2r-1,\ r \equiv 1 \text{(mod\ 4)};\\
\frac{(2^r-1)(2^{r-1}-1)}{3 \cdot 2^{r-1}} \; \zeta(r) \; \prod_{i=1}^{r-1} \frac{(2i -1)!}{(2 \pi)^{2i}} \zeta(2i),
& \text{if } n= 2r-1,\ r \equiv 3 \text{(mod\ 4)}.
\end{array}
\right.
\end{equation}
\smallskip

\noindent
(The fields $\ell_0$ and $\ell_1$ are defined by $\ell_1 = \QQ[\sqrt{-3}]$  and $\ell_0$ is the quartic field with a defining polynomial $x^4-x^3+2x-1$. The functions $\zeta_K(s)$, $L_{L|K}(s)$ and $\zeta(s)$ denote the Dedekind zeta function of a field $K$, the Dirichlet $L$-function associated to a quadratic field extension $L/K$, and the Riemann zeta function, respectively.)

The groups $\Gamma_0^n$ and $\Gamma_1^n$ can be described as the normalizers of stabilizers of integral lattices in quadratic spaces $(V, f)$,
$$ f = d x_0^2 + x_1^2 + \ldots x_n^2,$$
where for $\Gamma_0^n$ we take $d =  - \frac12(1+\sqrt{5})$ if $n$ is even and $d = (-1)^r 3 - 2\sqrt{5}$ if $n = 2r-1$ is odd, and for $\Gamma_1^n$ we have $d = -1$ except for the case when $n = 2r - 1$ is odd and $r$ is even where $d = -3$. In even dimensions the stabilizers of the lattices under consideration appear to be maximal discrete subgroups so the index $[\Gamma_i^{2r}:\Lambda_i^{2r}] = 1$ and $\Gamma_i^{2r}$ are principal arithmetic subgroups ($i = 0,1$). In the odd dimensional case the index is equal to $2$ except $[\Gamma^{2r-1}_1:\Lambda^{2r-1}_1] = 1$ when $r = 2m + 1$ with $m$ even. In the noncompact case the corresponding covolumes (without proof of minimality) were previously computed by Ratcliffe and Tschantz (cf. \cite{RatTsch97, RatTsch13}), who achieved this by explicitly evaluating the limit in the classical Siegel's volume formula \cite{Sieg35, Sieg36}.

The formulas \eqref{mbel:sec3:f1}--\eqref{mbel:sec3:f2} look scary but they are explicit and can be applied for computation and estimation of the volumes. As an example of such computation we can study the growth of minimal volume depending on the dimension of the space. Figure~\ref{mbel:graph} presents a graph of the logarithm of the minimal volume of compact/noncompact arithmetic orbifolds in dimensions $n < 30$. The logarithmic graph of the Euler characteristic in even dimensions has a similar shape. By analyzing this image we come up with several interesting corollaries which can be then confirmed analytically (as was done in \cite{Bel04} and \cite{BelEme}):

\begin{corollary}\label{mbel:sec3:cor}
The minimal volume decreases with n till $n = 8$ (resp. $n = 17$) in the compact (resp. noncompact) case. After this it starts to grow eventually reaching a very fast super-exponential growth. For every dimension $n\ge 5$, the minimal volume of a noncompact arithmetic hyperbolic $n$-orbifold is smaller than the volume of any compact arithmetic hyperbolic $n$-orbifold. Moreover, the ratio
between the minimal volumes $\Vol(\cO^n_0)/\Vol(\cO^n_1)$ grows super-exponentially with $n$.
\end{corollary}

\begin{figure}[!ht]
\centering
\psfig{file=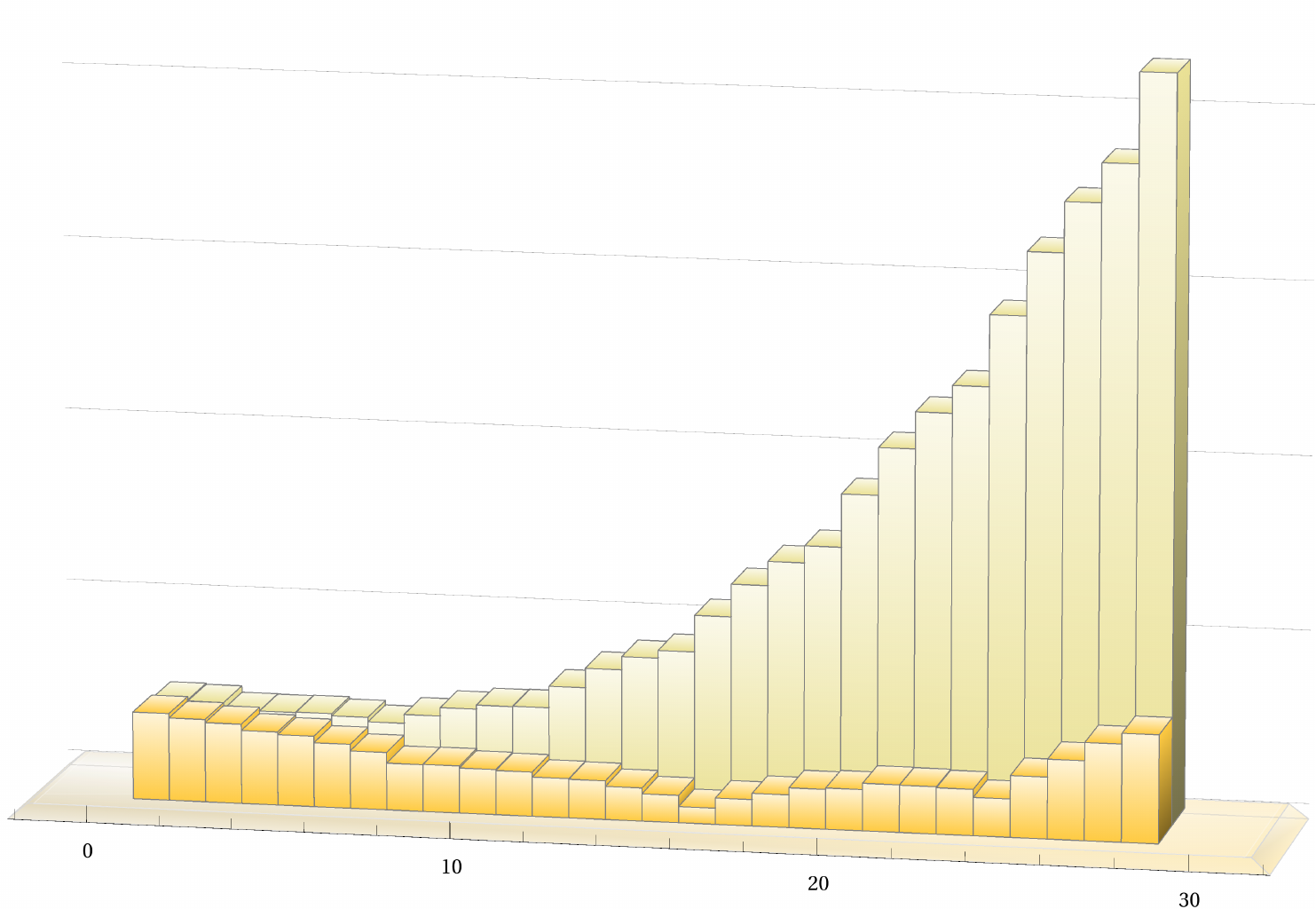, scale=0.75}
\caption{\label{mbel:graph} The logarithm of the minimal volume of noncompact (front) and compact (back) arithmetic hyperbolic $n$-orbifolds for $n = 2, 3, \ldots, 29$.}
\end{figure}

Another interesting corollary of the growth of minimal volume was obtained by Emery:
\begin{theorem}[Emery, \cite{Eme-Euler2}]
For $n > 4$ there is no compact arithmetic hyperbolic $n$-manifold $\cM$ with $|\chi(\cM)| = 2$.
\end{theorem}
In particular, there do not exist arithmetically defined hyperbolic rational homology $n$-spheres with $n$ even and bigger than $4$. We can remark that for $n > 10$ this theorem follows from the results in \cite{Bel04} pertaining to Corollary~\ref{mbel:sec3:cor}, but smaller dimensions are harder and require more careful analysis of the Euler characteristic of arithmetic subgroups.

\medskip

It is conjectured that all results in this section are true without assuming arithmeticity. We shall discuss this conjecture more carefully in Section~\ref{mbel:sec5}.

\section{Minimal volume manifolds and cusps}\label{mbel:sec4}

An interesting and somewhat surprising corollary of the results about the minimal volume arithmetic hyperbolic $n$-orbifolds is that for $n \ge 5$ the minimum is attained on \emph{noncompact} hyperbolic $n$-orbifolds. In a joint work with Emery we observed that the picture becomes even more interesting when we restrict our attention to manifolds. As a result we came up with a conjecture \cite{BelEme13}:

\begin{conjecture} \label{mbel:sec4:conj}
 Let $\mathcal{M}$ be a compact hyperbolic manifold of dimension $n\neq 3$. Then there exists a noncompact hyperbolic $n$-manifold $\mathcal{M}_1$ whose volume is smaller than the volume of $\mathcal{M}$.
\end{conjecture}

Dimension $n = 3$ is special because it is the only dimension in which we can perform hyperbolic Dehn fillings on the cusps of a finite volume noncompact manifold $\mathcal{M}_1$ to obtain compact hyperbolic manifolds of smaller volume. This follows from Thurston's Dehn surgery theorem (cf. \cite[Sections~5 and 6]{Thur80}). Our conjecture essentially says that this is the only way to produce very small compact hyperbolic $n$-manifolds.

The conjecture is known to be true in dimensions $n\le 4$ and $6$. More precisely, for these $n$ there exist noncompact hyperbolic $n$-manifolds $\cM$ with $|\chi(\cM)| = 1$ \cite{RatTsch00, ERT10}, whereas it is a general fact that the Euler characteristic of a compact hyperbolic manifold is even (cf. \cite[Theorem~4.4]{Ratcliffe2001899}). The main result of \cite{BelEme13} is a proof of the conjecture for arithmetic hyperbolic $n$-manifolds of dimension $n \ge 30$. In the next section we shall discuss the minimal volume conjecture (MVC), which together with this result would imply Conjecture~\ref{mbel:sec4:conj} for $n\ge 30$, however, the above conjecture is weaker than the MVC and we hope that it might be possible to attack it directly.

The proof of the theorem in \cite{BelEme13} is based on the results about minimal volume arithmetic hyperbolic $n$-orbifolds discussed in the previous section combined with a certain control over their manifold covers. For the latter we use explicit arithmetic constructions providing the upper bounds and the orbifold Euler characteristic for the lower bounds. One of our findings was that even in odd dimensions the Euler characteristic could provide an effective tool for bounding the degree of the smooth covers. Indeed, if an orbifold $\cO$ under consideration has an even dimensional totally geodesic suborbifold $\mathcal{S}$, then the denominator of $\chi(\mathcal{S})$ gives a lower bound for the degrees of the manifold covers of $\cO$. It appears that small volume hyperbolic $n$-orbifolds tend to contain many totally geodesic codimension-one suborbifolds whose Euler characteristics we can use.

The main feature of noncompact finite volume hyperbolic orbifolds is that they have infinite ends that are called \emph{cusps}. Any such cusp is diffeomorphic to $\mathcal{N}\times [0, +\infty)$ for some closed connected flat $(n-1)$-orbifold $\mathcal{N}$. Geometry of the cusps plays a major role in the study of noncompact hyperbolic orbifolds and their volumes. For example, the cusp volume was used by Meyerhoff in his work on the noncompact minimal volume hyperbolic $3$-orbifold \cite{Meyerh85}, and later by Hild in the proof of minimality for hyperbolic $n$-orbifolds in dimensions $n\le 9$ \cite{Hild07}.

Long and Reid showed that any closed flat $(n-1)$-manifold $\mathcal{N}$ is diffeomorphic to a cusp cross-section of a finite volume hyperbolic $n$-orbifold $\cM$ \cite{LongReid02}. It was later proved by McReynolds that the same can be achieved with $\cM$ being a manifold \cite{McReynolds09}. In both constructions the resulting $n$-orbifold or manifold can be chosen to be arithmetic. By Margulis lemma, any finite volume hyperbolic $n$-orbifold has a finite number of cusps (cf. \cite[Proposition~5.11.1]{Thur80}). In the same paper Long and Reid raised a question about existence of $1$-cusped hyperbolic $n$-manifolds for $n \ge 4$. In dimension $4$ this problem was recently solved by Kolpakov and Martelli \cite{KolpMart13}, who constructed infinite families of hyperbolic $4$-manifolds with any given number $k \ge 1$ of cusps. The method of Kolpakov--Martelli is specific for $n = 4$ and not applicable in higher dimensions. In particular, it is not known if there exists a $5$-dimensional $1$-cusped hyperbolic manifold. On the other hand, Stover \cite{Stover13} has shown that in dimensions $n\ge 30$ there are no $1$-cusped arithmetic hyperbolic $n$-orbifolds (or manifolds). Following our yoga, this suggests that there should be no any $1$-cusped hyperbolic $n$-orbifolds in high dimensions. To conclude the discussion, let us mention that (arithmetic) $1$-cusped hyperbolic $n$-orbifolds  in dimensions $n\le 9$ appear, for example, in Hild's paper \cite{Hild07}, and for $n = 10$ and $11$ were constructed by Stover in \cite{Stover13}. The existence of $1$-cusped orbifolds in dimensions $12 \le n \le 29$ is not known and there is not even a conjecture about it.

A careful reader would notice that dimension bound $30$ appeared in two independent results discussed in this section. It is also the dimension bound in the celebrated Vinberg's theorem that says that there no arithmetic hyperbolic reflection groups in dimensions $n \ge 30$ \cite{Vinb81}. There is no reason, however, to expect that any of these bounds is sharp. The coincidence can be explain by the fact that arithmetic methods work very well in higher dimensions and $30$ is about the place where this starts to be noticeable. It might be possible to push down the bounds using the same methods but it would require a considerable effort and obtaining a sharp bound for any of these problems would most likely require some totally new ideas.

\section{Minimal volume without arithmeticity}\label{mbel:sec5}

By the Kazhdan--Margulis theorem \cite{KM68} and the subsequent work of Wang \cite{Wang72}, we know that for $n\ge 4$ there exists a minimal volume hyperbolic $n$-orbifold. The classical results on uniformization of Riemann surfaces and classification of Fuchsian groups imply that the same holds true for $n = 2$, while the work of J{\o}rgensen--Thurston \cite[\S~5.12]{Thur80} implies the same for $n = 3$. These papers also imply that there are smallest representatives in the restricted classes of compact/noncompact orbifolds or manifolds. Thus for each dimension $n$ we have four positive numbers representing the minimal values in the volume spectra.

A \emph{folklore conjecture}, which we call the \emph{MVC}, says that the minimal volume is always attained on arithmetic quotient spaces. This conjecture was known for a long time for $n = 2$. (Note that the smallest volume for
compact or noncompact manifolds in dimension $2$ is attained also on nonarithmetic surfaces, and conjecturally this is the only dimension when it happens.) The MVC has now been completely confirmed for $n = 3$ --- see \cite{Meyerh85, Adams87, GMM10, GehrMart09, MarshMart12} for the results covering each of the four cases. The smallest noncompact hyperbolic $n$-orbifolds for $n \le 9$ were determined by Hild in his thesis \cite{HildPhD} (see also \cite{Hild07}) and they are all arithmetic. For $n = 4$ and $n= 6$ there are examples of noncompact arithmetic hyperbolic $n$-manifolds $\mathcal{M}$ with $|\chi(\mathcal{M})| = 1$ which is the smallest possible \cite{RatTsch00, ERT10}. The smallest known compact orientable hyperbolic $4$-manifolds have $\chi = 16$. They were constructed independently by Conder--Maclachlan \cite{CondMac05} and Long \cite{Long08} and can be described as finite-sheeted covers of the smallest compact arithmetic hyperbolic $4$-orbifold from \cite{Bel04}, but it is not known if there exist any smaller examples. In fact, the problem of finding a compact hyperbolic $4$-manifold of minimal volume was one of the main motivations for \cite{Bel04}. Most of the small dimensional examples discussed here are ultimately related to hyperbolic reflection groups and we refer to the survey paper by Kellerhals for more about this connection \cite{Kel14}. So far, these are the only known results supporting the conjecture.

In this section we are going to discuss known general lower bounds for the volume that do not require arithmeticity. These bounds come either from a quantitative analysis of the proof of the Kazhdan--Margulis theorem or from \emph{the Margulis lemma} and related estimates.

\medskip

Let us recall the Margulis lemma for the case of hyperbolic spaces (cf. \cite[Lemma 5.10.1]{Thur80}):
\begin{lemma}
 For every dimension $n$ there is a constant $\mu = \mu_n > 0$ such that for every discrete group $\Gamma < \mathrm{Isom}(\Hy)$ and every $x \in \Hy$, the group $\Gamma_\mu(x) = \langle \gamma\in\Gamma\;\mid\; \mathrm{dist}(x,\gamma(x)) \le \mu\rangle$ has an abelian subgroup of finite index.
\end{lemma}

For a given discrete group $\Gamma$, the maximal value of $\mu$ such that $\Gamma_\mu(x)$ is virtually abelian is called the \emph{Margulis number} of $\Hy/\Gamma$, and the constant $\mu_n$ from the lemma is called the \emph{Margulis constant} of $G = \mathrm{Isom}(\Hy)$.

If $\cM = \Hy/\Gamma$ is a manifold, this result allows us to define its decomposition $\cM = \cM_{(0,\mu]} \cup \cM_{[\mu,\infty)}$ into a thin and thick parts, and then to estimate from below the volume of $\cM$ by the volume $v(\epsilon)$ of a hyperbolic ball of radius $\epsilon = \mu/2$ which embeds into the thick part $\cM_{[\mu,\infty)}$. The case of orbifolds is much more delicate but still it is possible to use the Margulis lemma to give a lower bound for the volume. This was shown by Gelander in \cite{BGLS10}. The resulting bound for the volume of $\cO^n = \Hy/\Gamma$ is
\begin{equation}
 \Vol(\cO^n) \ge \frac{2v(0.25\epsilon)^2}{v(1.25\epsilon)}, \quad \epsilon = \mathrm{min}\{\frac{\mu_n}{10},1\}.
\end{equation}
The problem with this bound is that in higher dimensions we do not have a good estimate for the value of $\mu_n$. To my best knowledge the only appropriate general estimate can be found in \cite[p.~107]{BGS85}. It gives
\begin{equation}
 \mu_n \ge \frac{0.49}{16\left(1 + 2\left(\frac{4\pi}{0.49}\right)^{n(n-1)/2}\right)}.
\end{equation}
In \cite{Kel04}, Kellerhals gave a much better bound for the Margulis constant of hyperbolic $n$-manifolds but it is not clear if her result should extend to orbifolds.

\medskip

In connection with these results it would be interesting to understand how the Margulis constant depends on the dimension of the hyperbolic space. All known bounds for $\mu_n$ decrease to zero exponentially fast when $n$ goes to infinity, but does $\mu_n$ actually tend to zero? Note that if we define the arithmetic Margulis constant $\mu_n^a$ as the minimal value of the Margulis numbers of arithmetic hyperbolic $n$-orbifolds, then a positive solution to Lehmer's problem about the Mahler measure of algebraic numbers (cf. \cite{Smyth08}) would imply that there is a uniform lower bound for $\mu_n^a$. So conjecturally $\mu_n^a$ is bounded. The situation with $\mu_n$ is different as it is shown by the following result, which I learnt from M.~Kapovich:

\begin{proposition}\label{prop_mu_n}
 There exists a constant $C > 0$ such that $\mu_n \le \frac{C}{\sqrt{n}}$.
\end{proposition}
\begin{proof}
The argument is based on ideas from \cite{Kapovich05}.

Let us fix $\epsilon > 0$. We want to construct a discrete group of isometries $\Gamma < \mathrm{Isom}(\Hy)$ for which the Margulis number of $\Hy/\Gamma$ is less than $\epsilon$. Let $\Gamma = F_2 = \langle f,g \rangle$, a two-generator free group. We would like to define a $\Gamma$-invariant quasi-isometric embedding of the Cayley graph $T$ of $\Gamma$ into $\Hy$ such that the generators of $\Gamma$ act by isometries with small displacement. The graph $T$ is a regular tree of degree four whose edges can be labeled by the generators $f$, $g$ and their inverses (starting from the root of $T$). Let us map the root to $p_0 \in \Hy$ and embed the edge corresponding to $f$ as a geodesic segment $[p_0, p_1]$ of length $\epsilon$. Now choose a geodesic through $p_1$ orthogonal to $[p_0, p_1]$ and map the $g$-edge adjacent to $p_1$ as an $\epsilon$-segment $[p_1, p_2]$ of this geodesic. We can continue this process inductively each time choosing a geodesic which is orthogonal to the subspace containing the previously embedded edges. The process terminates when we get to $p_n$ as $n$ is the dimension of the space. Along the way we have defined the action of the generators $f$, $g$ on the points $p_0$, $p_1,\ldots, p_n$ which can be extended to isometry of $\Hy$. Now use these isometries to embed the rest of the tree. The construction gives an embedding $\rho: T\to \Hy$ and an isometric action of $\Gamma$ on $\Hy$ leaving invariant the image $\rho(T)$. It remains to check that $\rho$ is a quasi-isometric embedding.

We need to estimate the distance in $\Hy$ between the images of different vertices $x,y\in T$ and check that it satisfies the quasi-isometric property with respect to $\dist_T(x,y)$. Let $x = p_0$ and $y = p_m$ for some $m > n$ --- this is a typical case and all the other easily reduce to it. Let $b_i = \dist_{\Hy}(p_0, p_i)$, so $b_0 = \epsilon$ and $b_n = \dist_{\Hy}(p_0, p_n)$. By the hyperbolic Pythagorean theorem we have
$$\cosh(b_{i+1}) = \cosh(b_i)\cosh(\epsilon),$$
therefore, $\cosh(b_n) = (\cosh(\epsilon))^n$. We now can use the disjoint bisectors test (cf. \cite[Section~3]{Kapovich05}). If the length $b_n$ is bigger than a certain constant (which can be taken $=2.303$ as in the proof of Lemma~3.2 [loc. cit.]), then the bisectors of $[p_0,p_n]$ and $[p_n,p_{2n}]$ do not intersect and hence are separated by the distance $\delta = \delta(\rho) > 0$. Hence we have $\dist_{\Hy}(p_0, p_m) \ge \delta [m/n]$, as the geodesic joining $p_0$ and $p_m$ will have to intersect all the intermediate bisectors. It follows that $\rho$ is a quasi-isometry, provided
\begin{equation*}%\label{eq_mar}
b_n \ge 2.303.
\end{equation*}

It remains to apply \cite[Lemma 2.2]{Kapovich05}, which shows that the isometric action of $\Gamma$ on $\Hy$ is discrete and hence
$$\mu_n \le \mu(\Hy/\Gamma) = \epsilon.$$

We conclude with an estimate for the constants:
$$\cosh(b_n) = (\cosh(\epsilon))^n \ge \cosh(2.303).$$
When $\epsilon\to 0$, we have $\cosh(\epsilon)  \approx 1 + \frac{\epsilon^2}2$, so there exists $C>0$ (can take 
$$C = \sqrt{2\log(\cosh(2.303))} = 1.799\ldots$$ 
if $n$ is sufficiently large) such that if $\epsilon \ge \frac{C}{\sqrt{n}}$ then $\rho$ is a quasi-isometric embedding.
\end{proof}

The groups in Proposition~\ref{prop_mu_n} have infinite covolume. It is tempting to try a similar argument on non-arithmetic lattices with small systole. Such lattices can be obtained by the inbreeding construction found by Agol for $n=4$ \cite{Agol06} and generalized to higher dimensions in \cite{BelT11} and \cite{BHW11}. However, as it stands for now, it is not clear how to make this work and the problem remains open. The other important open problem is to find a better lower bound for $\mu_n$. We can speculate that some kind of polynomially decreasing lower bound should exist.

\medskip

The other approach is to bound the volume via quantitative version of the Kazhdan--Margulis theorem. It goes back to the paper \cite{Wang69} by Wang, who found an explicit lower bound for the radius of a ball embedded into a Zassenhaus neighborhood of a Lie group $G$ (with $G \cong \PO(n,1)^\circ$ in our case). Adeboye and Wei combined this bound with a bound for the sectional curvature of $G$ to obtain an explicit lower bound for the volume \cite{AW12}. Their main result is
\begin{equation}
 \Vol(\cO^n) \ge \frac{2^{[\frac{6-n}{4}]}   \pi^{[\frac{n}{4}]}(n-2)!(n-4)!\cdots 1}{(2+9n)^{[\frac{n^2+n}{4}]}\Gamma(\frac{n^2+n}{4})}\int_{0}^{\min[0.08\sqrt{2+9n},\pi]}\sin^{\frac{n^2+n-2}{2}}\rho\,\,\,d\rho.
\end{equation}
This lower bound decreases super-exponentially when $n$ goes to infinity and it is currently the best known general lower bound for the volume.

For the noncompact hyperbolic orbifolds and manifolds there is also another approach to bounding the volume. It is based on estimating the density of Euclidean sphere packings associated to cusps. It was used in the papers of Meyerhoff, Adams, and Hild that were mentioned above. For the case of arbitrarily large dimension $n$, Kellerhals applied this method to obtain the best available lower bound for the volume of non-compact hyperbolic $n$-manifolds \cite{Kel98}:
\begin{equation}
\Vol(\cM^n_1) \ge m \frac{2n}{n(n+1)} \nu_n,
\end{equation}
where $m$ is the number of cusps of the manifold $\cM^n_1$ and $\nu_n$ denotes the volume of the ideal regular simplex in $\Hy$. Note that by Milnor's formula for large $n$ we have $\nu_n \approx \frac{e\sqrt{n}}{n!}$ and hence again we have a lower bound that decreases super-exponentially with $n$.

\medskip

In conclusion, we see that for large dimensions there is a very large gap between the known (\emph{super-exponentially decreasing}) bound and the conjectural (\emph{super-exponentially increasing}) values of the minimal volume. This gap highlights our limited understanding of the complexity and structure of high-dimensional hyperbolic manifolds and orbifolds, especially the non-arithmetic ones. I hope that the future research will shed more light onto this problem.

\bibliographystyle{amsplain}
\bibliography{mbel}

\end{document}